\documentclass[12pt]{amsart}
\usepackage[foot]{amsaddr}
\usepackage[margin = 1in]{geometry}

\usepackage{amsmath,amsfonts,amssymb,amsthm,url}
\usepackage{color}
\usepackage{fancyhdr}
\usepackage[colorlinks=true,citecolor=black,linkcolor=black,urlcolor=blue]{hyperref}
\usepackage{bookmark}
\usepackage{verbatim}
\usepackage{float}
\usepackage[all]{xy}
\usepackage[demo]{graphicx}
\usepackage{caption}
\usepackage{subcaption}
\usepackage{mathtools}
\usepackage{enumitem}
\usepackage{youngtab}
\usepackage{tikz}
\usepackage{diagbox}

\newtheorem{theorem}{Theorem}[section]

\newtheorem{corollary}[theorem]{Corollary}
\newtheorem{lemma}[theorem]{Lemma}

\newtheorem{proposition}[theorem]{Proposition}

\theoremstyle{definition}
\newtheorem{definition}[theorem]{Definition}

\newtheorem{remark}[theorem]{Remark}
\newtheorem{claim}[theorem]{Claim}
\newtheorem{example}[theorem]{Example}
\newtheorem{problem}[]{Problem}

\title{Enumeration and Extensions of Word-representants}

 \author{Marisa Gaetz$^1$}
 \address{$^1$Massachusetts Institute of Technology}
\email{mgaetz@mit.edu}

 \author{Caleb Ji$^2$} 
 \address{$^2$Columbia University}
\email{caleb.ji@columbia.edu}

\begin{document}

\maketitle              

\begin{abstract}
Given a finite word $w$ over a finite alphabet $V$, consider the graph with vertex set $V$ and with an edge between two elements of $V$ if and only if the two elements alternate in the word $w$. Such a graph is said to be \textit{word-representable} or \textit{$11$-representable} by the word $w$; this latter terminology arises from the phenomenon that the condition of two elements $x$ and $y$ alternating in a word $w$ is the same as the condition of the subword of $w$ induced by $x$ and $y$ avoiding the pattern 11. In this paper, we first study minimal length words that word-represent graphs, giving an explicit formula for both the length and the number of such words in the case of trees and cycles. We then extend the notion of word-representability (or $11$-representability) of graphs to $t$-representability of graphs, for any pattern $t$ consisting of two distinct letters. We prove that every graph is $t$-representable for several classes of patterns $t$. Finally, we pose a few open problems. 

\end{abstract}

\section{Introduction} 
The theory of word-representable graphs provides an important way to associate graphs with words. Motivated by the study of Perkins semigroups by Kitaev and Seif in \cite{K0}, this topic has been the subject of much research since its inception in 2008. A major theme of this research has been the classification of word-representable graphs (see, for instance, \cite{K2}, \cite{K5}, and \cite{K6}). Other papers, such as \cite{K3} and \cite{K4}, have studied variants and extensions of the original notion of word-representability. 

In this paper, we begin by giving a brief review of the basic definitions and results in this field.  For a more thorough treatment, we refer the reader to \cite{K1}. In what follows, all graphs are taken to be simple and undirected.  


\begin{definition}[{\cite[Definition 1]{K1}}] \label{def:word-rep}
A graph $G = (V,E)$ is \textit{word-representable} if there exists a word $w$ over the alphabet $V$ such that for any $x,y \in V$, $x\neq y$, $xy \in E$ if and only if $x$ and $y$ alternate in $w$; that is, if and only if there are no two instances of $x$ or $y$ without an occurrence of the other letter in between.  We require that $w$ contains each letter of $V$ at least once. For such a graph, we say that $w$ \textit{word-represents} $G$ or that $w$ is a \textit{word-representant} for $G$.
\end{definition} 

It is important to note that, following Kitaev, a word-representant $w$ for a graph $G = (V,E)$ must contain each letter of $V$ at least once. As a result of this requirement, we do not need to worry about whether it makes sense for letters $x$ and $y$ to ``alternate" in $w$ when one (or both) of $x$ and $y$ does not appear in $w$. 

 A word $w$ is called \textit{$k$-uniform} if each letter in $w$ occurs exactly $k$ times. For example, 12332414 is a 2-uniform word over $V = \{ 1,2,3,4 \}$, while the 1-uniform words are precisely the permutations of $V$. 

\begin{definition}[{\cite[Definition 2]{K1}}] \label{def:uniform-rep}
A graph $G$ is \textit{$k$-word-representable} if it has a $k$-uniform representant $w$. 
\end{definition}

In fact, it has been shown that Definitions \ref{def:word-rep} and \ref{def:uniform-rep} are equivalent.

\begin{theorem}[\cite{K2}] \label{thm:equiv-defns}
A graph is word-representable if and only if it is $k$-word-representable for some $k$.
\end{theorem}

For a word $w$, we call the permutation obtained by removing all but the leftmost occurrence of each letter in $w$ the \textit{initial permutation} of $w$. Let $\pi(w)$ denote the initial permutation of $w$. Similarly, we call the permutation obtained by removing all but the rightmost occurrence of each letter in $w$ the \textit{final permutation} of $w$. Let $\sigma(w)$ denote the final permutation of $w$.  In addition to the initial and final permutations of a word, it will often be useful to consider the restriction of a word to some subset of the vertices it is defined on. To denote a word $w$ restricted so some letters $x_1, \ldots, x_m$, we write $w \vert_{x_1 \cdots x_m}$. For example, if $w=132435213$, then $w\vert_{12} = 1221$.


In \cite{K5}, Halld\'orsson, Kitaev, and Pyatkin study the minimal $k$ for which certain graphs are $k$-word-representable. In Section \ref{sec:minimal length}, we study the absolute minimal length word-representants of trees and cycles, where we do not require our word-representants to be $k$-uniform for any $k$. We calculate the lengths of these absolute minimal length word-representants and give a formula for the number of such representants. In Section \ref{sec-patterns}, we introduce a natural alternative notion of representability for graphs, motivated by the alternative notion of graph representability given by Kitaev in \cite{K4}. Finally, we present some open problems in Section \ref{sec:future work}.


\section{Enumeration of minimal length representants} \label{sec:minimal length}
By Theorem \ref{thm:equiv-defns}, the following notion introduced in \cite{K1} by Kitaev is well-defined.

\begin{definition}[{\cite[Definition 3]{K1}}]
Let $G$ be a word-representable graph. The \textit{representation number} of $G$ is the least $k$ such that $G$ is $k$-word-representable.
\end{definition}

In \cite{K1}, Kitaev goes on to study the class of graphs with representation number two, and in \cite{K5}, Halld\'orsson, Kitaev, and Pyatkin study the minimal $k$ for which certain graphs are $k$-word-representable. Rather than following these authors in investigating minimal length \textit{uniform} word-representants of graphs, we here choose to study \textit{absolute} minimal length word-representants. In other words, we investigate the minimal length word-representants of graphs, without the requirement that the representants be uniform. 

\begin{definition}
For a word-representable graph $G$, define $\ell(G)$ as the minimal length of a word $w$ that word-represents $G$. 
\end{definition}

\begin{definition}
For a word-representable graph $G$, define $n(G)$ as the number of words of length $\ell(G)$ that word-represent $G$. 
\end{definition}

We obtain the following bound for $\ell(G)$ of a general graph $G$.
\begin{theorem}
\label{thm:min-length-graph}
Let $G = (V,E)$ be a word-representable graph with connected components $\{ G_i = (V_i, E_i) \}_{i=1}^{k}$. Then 
$$ \ell(G) \le \sum_{i=1}^k (\ell(G_i) + |V_i|) - \max_{1\le j\le k}{|V_j|}.$$
\end{theorem}
\begin{proof}
For each $i \in \{ 1,2,\ldots, k \}$, let $w_i$ be a minimal length word-representant of $G_i$. Without loss of generality, let $V_k$ be a maximal component.  We claim that the word $w = w_1 \sigma (w_1) \; w_2 \sigma(w_2) \cdots w_{k-1} \sigma(w_{k-1}) \; w_k$ represents $G$ and has length $\sum_{i=1}^k (\ell(G_i) + |V_i|) - |V_k|$. For each $i \in \{ 1,2,\ldots, k \}$, we have that $w_i$ represents $G_i$ over the alphabet $V_i$ and that $w_i \sigma(w_i)$ represents $G_i$ over the alphabet $V_i$ (since appending $\sigma(w_i)$ to $w_i$ does not affect which letters alternate). Noting that the sets $V_1, V_2, \ldots, V_k$ are pairwise disjoint, we have that $w \vert_{V_i}$ represents $G_i$ for all $1 \leq i \leq k$. Furthermore, for all pairs $(i,j)$ satisfying $1 \leq i < j \leq k$, we have that every vertex in $V_i$ occurs at least twice in $w$ before any vertex in $V_j$ appears. Therefore, $w$ accurately encodes each connected component of $G$ as well as the information that there are no edges between different connected components of $G$. It follows that $w$ word-represents $G$, as desired. Finally, it is clear by construction that $w$ has length $\sum_{i=1}^k (\ell(G_i) + |V_i|) - |V_k|$.  
\end{proof}

\subsection{Trees}
In the case of trees, we find a precise value for $\ell(G)$.  First we prove a simple bound for all triangle-free graphs. 

\begin{lemma}
\label{lem:triangle-free}
Let $G$ be a triangle-free graph with $n$ vertices.  Then $\ell(G)\ge 2n-2$.
\end{lemma}
\begin{proof}
Let $V$ be the vertex set of $G$ and let $w$ be a minimal length word-representant of $G$. By definition of word-representant (see Definition \ref{def:word-rep}), we have that all elements of $V$ occur at least once in $w$. We claim that there are at most two elements of $V$ that occur only once in $w$. Suppose, for the sake of contradiction, that $x,y,z \in V$ each occur exactly once in $w$. Then any pair of letters chosen from $\{ x,y,z \}$ must alternate in $w$. Consequently, there are edges $xy,yz,xz \in E$ forming a triangle in $G$, contradicting the fact that $G$ is triangle-free. Therefore, at most two elements of $V$ occur only once in $w$, so $\ell(G) \geq 2(n-2) +2 = 2n-2$. 
\end{proof}

\begin{theorem}
\label{thm:min-length-tree}
Let $T = (V,E)$ be a tree with $n := |V| \geq 2$ vertices. Then $\ell (T)=2n-2$.
\end{theorem}
\begin{proof}
By Lemma~\ref{lem:triangle-free}, $\ell(T)\ge 2n-2$.  We now show by induction on $n$ that there is a word of length $2n-2$ that represents $T$. Observe that the word $12$ represents the unique tree on two vertices (namely, the tree $T_2=(V_2,E_2)$ defined by $V_2 = \{ 1,2 \}$ and $E_2 = \{ 12 \}$). In other words, the result holds for $n=2$. Assume that the result holds up to $n=k-1$, where $k \geq 3$. Let $T_{k} = (V_k,E_k)$ be any tree on $k$ vertices. Let $a$ denote a leaf of $T_{k}$, and let $b$ denote the vertex adjacent to $a$. By the inductive hypothesis, there is a word $w'$ of length $2k-4$ that represents the tree $T_k \setminus \{ a \}$ obtained from $T_k$ by removing the vertex $a$ and the edge $ab$. Now, replace one instance of $b$ in $w'$ with $aba$, and let $w$ denote the resulting word. 

We claim that $w$ word-represents $T_k$. Recall from the above argument that at most two elements of $V_k \setminus \{ a \}$ occur only once in $w'$. Since $w'$ has length $2k-4$, all of the other $k-3$ elements of $V_k \setminus \{ a \}$ must occur exactly twice in $w'$. In particular, there are at most two instances of $b$ in $w'$ (and hence in $w$). It follows that $a$ and $b$ alternate in $w$. Furthermore, $a$ clearly does not alternate with any other letters in $w$. Consequently, $w$ is a word-representant of $T$ with length $2k-2$. The theorem follows by induction. 
\end{proof}

\begin{figure} 
\centering
\begin{minipage}{.5\textwidth}
  \centering
	\begin{tikzpicture}[scale=0.10]
\tikzstyle{every node}+=[inner sep=0pt]
\draw [black] (35.3,-13.2) circle (3);
\draw (35.3,-13.2) node {$1$};
\draw [black] (28.5,-23.8) circle (3);
\draw (28.5,-23.8) node {$2$};
\draw [black] (42,-23.8) circle (3);
\draw (42,-23.8) node {$3$};
\draw [black] (42,-34.5) circle (3);
\draw (42,-34.5) node {$4$};
\draw [black] (33.68,-15.73) -- (30.12,-21.27);
\draw [black] (36.9,-15.74) -- (40.4,-21.26);
\draw [black] (42,-26.8) -- (42,-31.5);
\end{tikzpicture}
\caption{A tree on four vertices.}
\label{fig:tree-four-vertices}
\end{minipage}%
\begin{minipage}{.5\textwidth}
  \centering
   \begin{tikzpicture}[scale=0.096]
\tikzstyle{every node}+=[inner sep=0pt]
\draw [black] (37.7,-25.8) circle (3);
\draw (37.7,-25.8) node {$1$};
\draw [black] (37.7,-11.3) circle (3);
\draw (37.7,-11.3) node {$2$};
\draw [black] (51.7,-34) circle (3);
\draw (51.7,-34) node {$3$};
\draw [black] (24.5,-34) circle (3);
\draw (24.5,-34) node {$4$};
\draw [black] (37.7,-22.8) -- (37.7,-14.3);
\draw [black] (40.29,-27.32) -- (49.11,-32.48);
\draw [black] (35.15,-27.38) -- (27.05,-32.42);
\end{tikzpicture}
 \caption{The star graph $S_3$.}
    \label{fig:star-3}
\end{minipage}
\end{figure}

\begin{example} 
\label{ex:tree-min-rep}
Figure \ref{fig:tree-four-vertices} depicts a tree on four vertices. According to Theorem \ref{thm:min-length-tree}, a minimal length representant for this tree has length $2 \cdot 4 - 2 = 6$. It is straightforward to check that 212434 is a minimal length representant of the depicted tree. 
\end{example}

For a tree $T = (V,E)$, we would now like to count the number $n(T)$ of minimal length word-representants for $T$. By Theorem \ref{thm:min-length-tree}, such representants have length $2|V|-2$ and are such that two vertices occur once and all other vertices occur twice. Note that the two vertices appearing once alternate and thus are connected in the graph.  Although this gives us some idea as to the structure of minimal length word-representants of $T$, we would like a more detailed picture. To this end, we establish the following notation and lemma.

For a tree $T = (V,E)$ and an edge $xy \in E$, let $T_{x, xy}$ be the subtree of $T$ obtained by deleting the edge $xy$ and taking the connected component containing $x$. Define $T_{y,xy}$ similarly.

\begin{lemma} \label{lem:subtree-order}
Let $T$ be a tree, and let $w$ be a minimal length representant for $T$. Let $x,y \in V(T)$ be the vertices of $T$ that occur only once in $w$; without loss of generality, let $x$ occur before $y$ in $w$. Then $w$ is of the form $w = w_{x_1} w_{x_2} \cdots w_{x_m} \; w_{y_1} w_{y_2} \cdots w_{y_n}$, where $w_{x_1}, w_{x_2}, \ldots, w_{x_m} \in T_{x,xy}$ and $w_{y_1}, w_{y_2}, \ldots, w_{y_n} \in T_{y,xy}$. 
\end{lemma}
\begin{proof}
Suppose, for the sake of contradiction, that not all elements of $T_{x,xy}$ occur before those of $T_{y,xy}$ in $w$. Then there exists some $v \in T_{y,xy}$ occurring before some $u \in T_{x,xy}$ in $w$. By our assumption that $x$ occurs before $y$ in $w$, it is not the case that both $u=x$ and $v = y$. Since $xy$ is the only edge of $T$ connecting a vertex in $T_{x,xy}$ to a vertex in $T_{y,xy}$, it follows that $u$ and $v$ do not alternate in $w$.

We first show that $u \neq x$. To this end, suppose, for the sake of contradiction, that $u = x$. Then $v \neq y$, meaning $v$ occurs twice in $w$. Using that $v$ and $x$ do not alternate in $w$, we have that $w \vert_{xvy} = vvxy$. Now, since $v \in T_{y,xy}$ and since $v$ does not alternate with $y$, there exists a path $v v_1 v_2 \cdots v_k y$ (of length at least 2) from $v$ to $y$ in $T$. Using the following claim, we will obtain a contradiction.

\begin{claim} \label{cl:v-placement}
If $u=x$, then for all $v_i \in \{ v_1, \ldots, v_k \}$, both instances of $v_i$ are to the left of $x$ in $w$.
\end{claim}

\begin{proof}[Proof of Claim \ref{cl:v-placement}]
We prove the claim by induction on $k$. Suppose that $k=1$, and observe that $v_1$ must alternate with $v$ but not with $x$. To alternate with $v$, there must be exactly one occurrence of $v_1$ between the two $v$'s in $w$. Then for $v_1$ not to alternate with $x$, the other occurrence of $v_1$ must also occur to the left of $x$ in $w$. In other words, either $w \vert_{v v_1 x y} = v_1 v v_1 v xy$ or $w \vert_{v v_1 xy} = v v_1 v v_1 xy$.

Assume now that both instances of each of $v_1, \ldots, v_{i-1}$ occur to the left of $x$ in $w$. Note that $v_i$ alternates with $v_{i-1}$ but not with $x$. By the same reasoning as in the base case, we see that both occurrences of $v_i$ are to the left of $x$ in $w$. The claim follows by induction.
\end{proof}


\noindent \textit{Proof of Lemma \ref{lem:subtree-order} (continued).} By the above claim, we have that both occurrences of $v_k$ are to the left of $x$ in $w$. Consequently, both occurrences of $v_k$ are to the left of $y$ in $w$, meaning $v_k$ and $y$ do not alternate. This contradicts our assumption that $v v_1 v_2 \cdots v_k y$ is a path in $T$. It follows by contradiction that $u \neq x$.

Given that $u \neq x$, we now consider each of the possibilities for the location(s) of $v$'s, and we show that each possibility leads to a contradiction. Let $u'$ be the rightmost element of $T_{x,xy}$ in $w$. By our assumption that $v$ occurs to the left of some $u \in T_{x,xy}$ in $w$, we have that $v$ occurs to the left of $u'$. Moreover, we have by the above that each instance of $v$ occurs to the right of $x$ in $w$. In particular, $u' \neq x$, so there are two instances of $u'$ in $w$. Consequently, there are four possibilities for $w \vert_{xvu'}$: 
\begin{equation*} 
w \vert_{xvu'} \in \{ xu'vvu' , \; u'xvvu', \; xvvu'u',  \; xvu'u'v \}.
\end{equation*}

Suppose that $w \vert_{xvu'} \in \{ xu'vvu' , \; u'xvvu' \}$. If $v=y$, then we would have $w \vert_{xvu'} \in \{ xu'yu', \; u'xyu' \}$, contradicting the fact that $u'$ does not alternate with $y$ in $w$. Hence, $v \neq y$. As before, let $v v_1 v_2 \cdots v_k y$ denote the path in $T$ from $v$ to $y$. By an inductive argument similar to the proof of Claim \ref{cl:v-placement}, we have that both instances of each of $v_1, v_2, \ldots, v_k$ occur between the two instances of $u'$ in $w$. Since $y$ alternates with $v_k$, $y$ must occur between the two instances of $v_k$, and hence between the two instances of $u'$. It follows that $y$ alternates with $u'$, a contradiction. 

Now, suppose that $w \vert_{xvu'} = xvvu'u'$. Since $u'$ does not alternate with $x$ in either of these cases, there exists a path $u' u_1' u_2' \cdots u_j' x \in T$ of length at least two. Observe that each $u_i' \in \{ u_1', u_2', \ldots, u_{j-1}' \}$ alternates with $u_{i-1}'$ but not with $x$ nor with $v$, and that $u_{j}'$ alternates with both $u_{j-1}'$ and with $x$. Using this observation, it is straightforward to verify that there exists $u_i' \in \{ u_1', u_2', \ldots, u_{j}' \}$ such that $w \vert_{xvu_i'} \in \{ xu_i' vv u_i', u_i' xvvu_i' \}$. Applying the previous case, we see that this is a contradiction.

Finally, suppose that $w \vert_{xvu'} = xvu'u'v$. As before, let $u' u_1' u_2' \cdots u_j' x$ denote the path from $u'$ to $x$ in $T$. Observe that $u_i'$ does not alternate with $v$ for all $u_i' \in \{ u_1', u_2', \ldots, u_j' \}$. Using this observation, we have that for all $u_i' \in \{ u_1', u_2', \ldots, u_j' \}$, both instances of $u_i'$ occur between the two instances of $v$ in $w$. This contradicts the fact that $u_j'$ alternates with $x$, completing the proof. 
\end{proof}

\begin{example}
\label{ex:tree-min-rep2}
Consider again the minimal length word-representant $w = 212434$ of the tree depicted in Figure \ref{fig:tree-four-vertices}. Using the notation of Lemma \ref{lem:subtree-order}, we have $x = 1$ and $y = 3$. As Lemma \ref{lem:subtree-order} predicts, all occurrences of the vertices of $T_{x,xy}$ (i.e., 1 and 2) occur before all occurrences of the vertices of $T_{y,xy}$ (i.e., 3 and 4) in $w$.
\end{example}

Lemma \ref{lem:subtree-order} gives us a lot of information regarding the structure of the minimal length word-representants of $T$. In fact, with the help of the next lemma, it will allow us to compute $n(T)$. This next lemma establishes $n(S_k)$, where $S_k$ is the \textit{star graph} with $k$ leaves (i.e., the complete bipartite graph $K_{1,k}$).

\begin{lemma}
\label{lem:star}
Let $S_k$ be the star graph with leaves $u_1, u_2, \ldots, u_k$ and with center vertex $c$. Then $n(S_k) = 2 \cdot k!$. Furthermore, any minimal length word-representant $w$ of $S_k$ takes one of the following forms:
$$w = u_k \; u_1 u_2 \cdots u_{k-1} \; c \; u_{k-1} u_{k-2} \cdots u_1 \hspace{.25cm} \text{ or } \hspace{.25cm} w = u_1 u_2 \cdots u_{k-1} \; c \; u_{k-1} u_{k-2} \cdots u_1 \; u_k.$$
\end{lemma}

\begin{proof}
Let $w$ be a minimal length word-representant for $S_k$. By Theorem \ref{thm:min-length-tree}, $w$ has length $2(k+1)-2 = 2k$. Since no two leaves of $S_k$ are adjacent, at most one leaf of $S_k$ appears only once in $w$. For $w$ to have length $2k$, it must be the case that this leaf (call it $u_k$) and the central vertex (call it $c$) occur exactly once in $w$, while the remaining $k-1$ leaves each occur twice. For these remaining leaves to alternate with $c$ in $w$, they must each occur in $w$ once before and once after $c$. Let $u_1 u_2 \cdots u_{k-1}$ be the order in which these leaves occur in $w$ before $c$. In order for no two $u_i$'s to alternate in $w$, these leaves must occur in the order $u_{k-1} u_{k-2} \cdots u_1$ after $c$. Finally, for $u_k$ not to alternate with any of the $u_i$'s, we see that either $w = u_k \; u_1 u_2 \cdots u_{k-1} \; c \; u_{k-1} u_{k-2} \cdots u_1$ or $w = u_1 u_2 \cdots u_{k-1} \; c \; u_{k-1} u_{k-2} \cdots u_1 \; u_k$.

There are $k$ choices for the leaf $u_k$ occurring only once in $w$; there are two choices for the position of $u_k$ (i.e., the beginning and the end of $w$). Moreover, there are $(k-1)!$ ways to assign values to the variables $u_1, u_2, \ldots , u_{k-1}$. Consequently, $S_k$ has $2 \cdot k!$ minimal length word-representants.   
\end{proof}

\begin{example}
Figure \ref{fig:star-3} shows the star graph $S_3$. Using the notation of Lemma \ref{lem:star}, $c=1$ for this graph. According to Lemma \ref{lem:star}, the depicted graph has the 12 minimal length representants given in Table \ref{tab:star-min}. 
\end{example}

\begin{table}
\begin{center}
\begin{tabular}{c|c|c}
    $u_k = 2$ & $u_k = 3$ & $u_k = 4$ \\ \hline
     234143 & 324142 & 432123 \\
     341432 & 241423 & 321234 \\
     243134 & 342124 & 423132 \\
     431342 & 421243 & 231324
\end{tabular}
\end{center}
\caption{Minimal length representants of the star graph $S_3$ (see Figure \ref{fig:star-3}).}
\label{tab:star-min}
\end{table}

With Lemmas \ref{lem:subtree-order} and \ref{lem:star}, we are now ready to compute $n(T)$. 
\begin{theorem}
\label{n(T)}
Let $T = (V,E)$ be a tree on at least two vertices. Then 
\[n(T) = 2\prod_{v\in V} \deg(v)!\sum_{xy\in E} \dfrac{1}{\deg(x) \deg(y)},\]
where $\deg(v)$ denotes the degree of $v$ in $T$.
\end{theorem}

\begin{proof}
As previously discussed, in any minimal length word-representant for $T$, there are two elements of $V$ that occur only once (call them $x$ and $y$), while all other elements of $V$ occur exactly twice. Note that in order for $x$ and $y$ to occur only once in a word-representant, they must be adjacent in $T$. In other words, the choice of such a pair of elements $(x,y)$ corresponds exactly to the choice of an edge $xy \in E$. 


In what follows, we will establish the number $n_x$ of word-representants of $T_{x,xy}$ of length $\ell (T_{x,xy}) + 1$, and the number $n_y$ of word-representants of $T_{y,xy}$ of length $\ell (T_{y,xy}) +1$. These word-representants can be arranged in two ways (determined by which of $x$ and $y$ occurs first) to obtain word-representants for $T$ of length $\ell (T) = 2n-2$. By Theorem \ref{thm:min-length-tree}, these representants are minimal. Moreover, by Lemma \ref{lem:subtree-order}, every minimal length representant of $T$ is of this certain form. Summing over $E$ to account for the possible choices of $(x,y)$, we see that 
\begin{equation} \label{eq:outline}
n(T) = 2 \cdot \sum_{xy \in E} (n_x \cdot n_y).
\end{equation}

To derive an expression for $n_x$, we will use an inductive method similar to breadth first search; this same method can be applied to obtain an expression for $n_y$. 

Let $N(x)$ denote the neighborhood of $x$ in $T$, and let $N'(x):= N(x) \setminus \{ y \}$ denote the neighborhood of $x$ in $T_{x,xy}$. A straightforward consequence of Lemma \ref{lem:star} is that $N'(x)$ has $(\deg (x)-1)!$ word-representants in which $x$ occurs exactly once and each element of $N'(x) \setminus \{ x \}$ occurs exactly twice. Using the notation of Lemma \ref{lem:star}, these are the word-representants of the form 
$$w = u_{\deg(x)-1} \; u_1 u_2 \cdots u_{\deg(x)-2} \; x \; u_{\deg(x)-2} u_{\deg(x)-3} \cdots u_1 \; u_{\deg(x)-1}.$$ 

Choose one of the aforementioned $(\deg (x)-1)!$ word-representants of $N'(x)$; call it $w_{N'(x)}$. We proceed by establishing the following claim.

\begin{claim} \label{cl:neighbors}
Let $z \in N'(x) \setminus \{x \}$. There are $\deg(z)!$ ways to insert the elements $\{ v_1, v_2, \ldots,\allowbreak v_{\deg(z)-1} \} = N(z) \setminus \{ x,z \}$ into $w_{N'(x)}$ to form length-$(\ell (N'(x) \cup N(z)) + 1)$ representants for the subtree $N'(x) \cup N(z)$ of $T_{x,xy}$. Moreover, these representants take the form
\begin{align*} 
\star \; \star \; \star \; v_{i_1} v_{i_2} \cdots v_{i_m}\; &z\; v_{i_m} v_{i_{m-1}} \cdots v_{i_1}\; \star \; \star \; \star \; x \; \star \; \star \;\; \star \\
& \star \; \star \; \star \; v_{i_{m+1}} v_{i_{m+2}} \cdots v_{i_{\deg(z)-1}} \; z \; v_{i_{\deg(z)-1}} v_{i_{\deg(z)-2}} \cdots v_{i_{m+1}} \; \star \; \star \;\; \star,     
\end{align*}
where the $\star$'s indicate an arbitrary number of elements of $N'(x) \setminus \{x\}$, and where 
$$\{ v_{i_1}, v_{i_2}, \ldots, v_{i_{\deg (z) -1}} \} = \{ v_1, v_2, \ldots, v_{\deg(z)-1} \}.$$
\end{claim}

\begin{proof}[Proof of Claim \ref{cl:neighbors}]
Let $w_G$ denote a word-representant of the graph $G$.  We proceed by induction on the number of neighbors of $z$. If $\deg (z) = 1$, we have $N(z) \setminus \{ x,z \} = \varnothing$, so there is nothing to show. Suppose that $\deg(z) = 2$. Let $n_1$ be the single element of $N(z) \setminus \{ x,z \}$. Of the vertices represented in $N'(x)$, $v_1$ is adjacent only to $z$. Moreover, any $s \in N'(x) \setminus \{ x,z \}$ alternates with $x$ but not with $z$ or $v_1$. Therefore, to obtain a word-representant of $N'(x) \cup N(z)$ of length $\ell (N'(x) \cup N(z))+1$, we must insert $v_1$ twice into $w_{N'(x)}$ so that we get a word $w_{N'(x) \cup N(z)}$ satisfying
$$w_{N'(x) \cup N(z)} \vert_{xzsv_1} \in \{ sv_1zv_1xzs, \; v_1zv_1sxsz \} \; \forall s \in N'(x) \setminus \{ x,z \}.$$
In other words, the two instances of $v_1$ must be placed so as to immediately surround one of the instances of $z$. Therefore, there are $2 = 2! = \deg(z)!$ ways to insert the element of $N(z) \setminus \{ x,z \}$ into $w_{N'(x)}$ to form word-representants for $N'(x) \cup N(z)$ of the desired length, and these possibilities both have the desired form. 

Having established our base case, let us assume that the result holds for up to $k-1$ neighbors of $z$, and suppose $\deg (z) = k$. By this inductive assumption, there are $(k-1)!$ ways to insert the vertices $v_1, \ldots, v_{k-2}$ into $w_{N'(x)}$ to form word-representants $w_{N'(x) \cup \{ v_1,\ldots, v_{k-2} \}}$ of the desired length; moreover, these representants take the form described above. Therefore, it suffices to show that there are $k$ ways to insert $v_{k-1}$ into a given representant $w_{N'(x) \cup \{ v_1,\ldots, v_{k-2} \}}$ to get a representant $w_{N'(x) \cup N(z)}$ of length $\ell (N'(x) \cup N(z))+1$, and that the resulting representants are of the desired form.  

Suppose we are given a representant $w_{N'(x) \cup \{ v_1,\ldots, v_{k-2} \}}$ of length $\ell (N'(x) \cup \{ v_1,\ldots, \allowbreak v_{k-2} \})+1$ and the form 
\begin{align*}
    \star \; \star \; \star \; v_{i_1} \cdots v_{i_m} \; z \; v_{i_m} \cdots v_{i_1} \; \star \; \star \; \star \; x \; \star \; \star \; \star \; v_{i_{m+1}} \cdots v_{i_{k-2}} \; z \; v_{i_{k-2}} \cdots v_{i_{m+1}} \; \star \; \star \;\; \star, 
\end{align*}
where again the $\star$'s are elements of $N'(x) \setminus \{x \}$. Recall that any element of $N'(x) \setminus \{ x \}$ alternates only with $x$ and that we want to insert $v_{k-1}$ so that it alternates with $z$ but not with anything else. With this in mind, it is straightforward to see that the two $v_{k-1}$'s must be placed either immediately surrounding one of the $z$'s or immediately surrounding a pair of $v_{i_j}$'s. Hence, there are $m+1$ possible placements among $\star \; \star \; \star \; v_{i_1} \cdots v_{i_m} \; z \; v_{i_m} \cdots v_{i_1} \; \star \; \star \; \star$ and $k-m-1$ possible placements among $\star \; \star \; \star \; v_{i_{m+1}} \cdots v_{i_{k-2}} \; z \; v_{i_{k-2}} \cdots v_{i_{m+1}} \; \star \; \star \; \star$; together, there are $k$ possible placements of the $v_{k-1}$'s, giving the desired result.  
\end{proof}

\noindent \textit{Proof of Theorem \ref{n(T)} continued.}
We now proceed by continuing to describe our breadth-search method for computing $n_x$. To recap, we have started by considering $x$ as well as all of its neighbors in $T_{x,xy}$. As mentioned above, there are $(\deg(x)-1)!$ word-representants of $N'(x)$ containing one instance of $x$ and two instances of everything else. 

Next, we consider the neighbors of each vertex of $N'(x) \setminus \{ x \}$. (Under the breadth-first search analogy, we would like to add these neighbors to a queue.) Claim \ref{cl:neighbors} shows that for any such $z \in N'(x) \setminus \{ x \}$, there are $\deg(z)!$ ways to add two instances of each element of $N(z) \setminus \{x,z \}$ to a given word-representant of $N'(x)$ to get a word-representant of $N'(x) \cup N(z)$. Replacing $x$ with $z$, and considering some $t \in N(z) \setminus \{ x,z\}$, Claim \ref{cl:neighbors} gives that there are $\deg (t)!$ ways to add two instances of each element of $N(t) \setminus \{ z,t \}$ to a given word-representant of $N'(x) \cup N(z)$ to get a word-representant of $N'(x) \cup N(z) \cup N(t)$. Continuing in this manner, we see that there are $(\deg (x) - 1)! \prod_{v \in V(T_{x,xy})} \deg (v)!$ word-representants of $T_{x,xy}$ of length $\ell (T_{x,xy})+1$. Hence we have shown that
$$n_x = (\deg (x) - 1)! \prod_{u \in V(T_{x,xy})} \deg (u)! \hspace{1cm} \text{and} \hspace{1cm} n_y = (\deg (y) - 1)! \prod_{v \in V(T_{y,xy})} \deg (v)!.$$
Looking back at Equation (\ref{eq:outline}), it follows that
$$n(T) = 2 \cdot \sum_{xy \in E} \frac{\prod_{v \in V} \deg (v)!}{\deg (x) \deg (y)} = 2 \prod_{v \in V} \deg (v)! \sum_{xy \in E} \frac{1}{\deg (x) \deg (y)},$$
as desired.
\end{proof}

\begin{example}
Consider the tree $T = (V,E)$ on four vertices depicted in Figure \ref{fig:tree-four-vertices}. According to Theorem \ref{n(T)}, 
\begin{align*}
    n(T) = 2 & \prod_{v \in V} \deg (v)!  \sum_{xy \in E} \frac{1}{\deg (x) \deg (y)} \\
    &= 2 \left ( 2! \cdot 1! \cdot 2! \cdot 1! \right ) \left ( \frac{1}{1 \cdot 2} + \frac{1}{2 \cdot 2} + \frac{1}{2 \cdot 1} \right ) = 8 \left ( \frac{5}{4} \right ) = 10.
\end{align*}
Indeed, it is straightforward to verify that the ten representants shown in Table \ref{tab:tree-min-reps} are precisely the minimal word-representants of $T$, where $x$ and $y$ are the vertices appearing once.

\begin{table}
    \begin{center}
    \begin{tabular}{c|c}
        $\{ x,y \}$ & Corresponding Minimal Length Representants of $T$ \\ \hline
        $\{ 1,2 \}$ & 231434, 314342, 243413, 434132 \\ 
        $\{ 1,3 \}$ & 212434, 434212 \\ 
        $\{ 3,4 \}$ & 212314, 132124, 421231, 413212
    \end{tabular}
    \caption{Minimal length representants of the tree shown in Figure \ref{fig:tree-four-vertices}.}
    \label{tab:tree-min-reps} 
    \end{center}
\end{table}
\end{example}
Theorem \ref{n(T)} gives the following corollary.

\begin{corollary}
Let $P_k$ be the path on $k$ vertices. Then for $k\ge 3$, we have 
$$n(P_k) = (k+1)\cdot 2^{k-3}.$$
\end{corollary}

\subsection{Cycles}

In this subsection, we consider the cycle graph $C_n$ on $n \geq 3$ vertices. The minimal representants of $C_3$ are quite easy to understand.

\begin{proposition}
We have $\ell (C_3) = 3$ and $n(C_3) = 6$.
\end{proposition}

\begin{proof}
Since a word-representant by definition contains each vertex at least once, $\ell (C_3) \geq 3$. It is then straightforward to see that the six permutations 123, 132, 213, 231, 312, 321 are the only length-three word-representants of $C_3$, verifying that $\ell (C_3) = 3$ and that $n(C_3) = 6$.
\end{proof}

For $n \geq 4$, the minimal representants of $C_n$ are more complicated. 

\begin{theorem} \label{thm:cycle-rep-length}
Let $C_n = (V,E)$ be the cycle graph on $n \geq 4$ vertices. Then $\ell (C_n)=2n-2$.
\end{theorem}
\begin{proof}
Note that for $n \geq 4$, $C_n$ is triangle-free; therefore, Lemma~\ref{lem:triangle-free} gives that $\ell(C_n)\ge 2n-2$. We now show that $\ell (C_n) \leq 2n-2$ by constructing a length $2n-2$ word-representant for $C_n$. Suppose that the vertices of $C_n$ are labeled $1,2,\ldots, n$ such that there are edges between vertices labeled with consecutive integers, as well as between $n$ and 1. Consider the word $w' = n1 \; (n-1)n \; (n-2)(n-1) \; (n-3)(n-2) \; \cdots \; 45 \; 34 \; 23$. Here, we format $w'$ by adding space between every pair of letters to illustrate the structure of the word; the first letters in each pair form the decreasing sequence $n,n-1,n-2,n-3, \ldots, 4,3,2$, while the second letters in each pair form the sequence $1,n,n-1,n-2, \ldots, 5,4,3$. In this construction, the letters 1 and 2 appear exactly once, while the letters $3,4,\ldots, n$ appear exactly twice. Moreover, for every vertex $v \in V \setminus \{ 1,2 \}$, the letters appearing between the two occurrences of $v$ in $w'$ are precisely $v+1$ and $v-1$. Consequently, $v$ alternates only with $v+1$ and $v-1$ in $w'$. Additionally, 1 alternates only with 2 and $n$, and 2 alternates only with 1 and 3.
\end{proof}

\begin{example}
Consider the cycle graph $C_5$ shown in Figure \ref{fig:five-cycle}. According to Theorem \ref{thm:cycle-rep-length}, $\ell (C_5) = 2\cdot 5-2 = 8$. It is straightforward to check that 51453423 is therefore a minimal length word-representant of $C_5$.  
\end{example}

Having established that $\ell (C_3) = 3$ and that $\ell (C_n) = 2n-2$ for $n \geq 4$, we would now like to establish $n(C_n)$. We have already shown that $n(C_3) = 6$. In fact, $n(C_n) = 2n$ for all $n \geq 3$. 

\begin{theorem}
\label{thm:cycle-rep-number}
Let $C_n = (V,E)$ be the cycle graph on $n$ vertices. For $n\ge 3$, $n(C_n)=2n$.
\end{theorem}

\begin{proof}
Let $w$ be a minimal length word-representant for $C_n$. Since $\ell(C_n) = 2n-2$, two elements of $V$ will occur exactly once in $w$, while the remaining vertices will occur exactly twice. 

There are $n$ choices to choose the pair of vertices which appear only once, since they alternate and thus there must be an edge between them. After this pair of vertices is chosen, there are two orders in which they can occur in a minimal length representant. Without loss of generality, suppose the pair we have chosen is $(1,2)$ and that 1 occurs before 2.

We will now consider the ways in which two instances of each of $3,4,\ldots, n$ can be placed in this minimal word-representant $w$ (which is assumed to be such that $w \vert_{12} = 12$. Certainly, the two 3's must surround the 2 but not the 1. Moreover, the two 4's must alternate with the two 3's without alternating with the 1 or the 2. We claim that the 4's must be between 1 and 2. Indeed, otherwise both 4's will be to the right of the 2. If this is the case, then every successive number must have at least one instance to the right of the 2, and hence both instances to the right of the 2 (so as to avoid alternating with 2). This would make it impossible for $n$ to alternate with 1. Therefore, our representant must satisfy $w \vert_{1234} = 1 4 3 4 2 3$. Similarly, the two 5's must surround the leftmost 4 without alternating with 3, giving $w \vert_{12345} = 1 545 3423 $. We can continue in this manner and conclude with the two $n$'s immediately surrounding 1 and the leftmost instance of $n-1$. 

Observe that after choosing the pair $(1,2)$ and their relative locations, this process has involved no choices. In other words, once $1$ and $2$ were chosen and placed in that order, there was only one way to form a word-representant of $C_n$. It follows that $n(C_n) = 2n$, as desired.
\end{proof}

\begin{remark}
We invite the reader to compare Theorem \ref{thm:cycle-rep-number} with \cite[Theorem 1]{4n}, which says that for $n > 3$, the cycle graph $C_n$ has exactly $4n$ $2$-uniform word-representants. 
\end{remark}

\section{Graphs representable from pattern avoidance in words}
\label{sec-patterns}

So far in this paper we have exclusively considered word-representable graphs. In this section, we define a more general notion of representability for graphs, motivated by Kitaev in \cite{K4}. To this end, we first establish two preliminary definitions.

\begin{definition}
Two words are said to be \textit{isomorphic} if there is a bijective, pattern-preserving correspondence between their letters.
\end{definition}

\begin{example}
The words 112134, 332378, and $aabacd$ are all isomorphic. 
\end{example}

\begin{definition}
Let $w$ be a word defined on an alphabet $V$, and let $u$ be a word defined on an alphabet $U$. Then $w$ \textit{contains} $u$ if there exists a subset $\{ x_1, x_2, \ldots, x_{\vert U \vert } \} \subseteq V$ such that $w \vert_{x_1, x_2, \ldots, x_{\vert U \vert}}$ has a contiguous subword isomorphic to $u$. If $w$ does not contain $u$, we say $w$ \textit{avoids} $u$.

\end{definition}

\begin{example}
Let $w = 121223$ and let $u = 112$. Then $w$ contains $u$, since $w \vert_{23} = 2223$ has a subword that is isomorphic to $u$ (namely, 223).
\end{example}

With these definitions in mind, we can now introduce a generalized notion of graph representability.

\begin{definition} \label{def:t-rep}
Let $t$ be a word on two letters. A graph $G = (V,E)$ is $t$-representable if there exists a word $w$ over the alphabet $V$ such that for any $x,y \in V$, $xy \in E$ if and only if $w \vert_{xy}$ avoids the pattern given by $t$. We require that $w$ contains each letter of $V$ at least once, and we say that $w$ $t$-represents $G$, or that $w$ is a $t$-representant for $G$.  
\end{definition}

\begin{figure}
\centering
\begin{minipage}{.5\textwidth}
  \centering
  \begin{tikzpicture}[scale=0.1]
\tikzstyle{every node}+=[inner sep=0pt]
\draw [black] (40.2,-17.5) circle (3);
\draw (40.2,-17.5) node {$1$};
\draw [black] (51.5,-26) circle (3);
\draw (51.5,-26) node {$2$};
\draw [black] (47.3,-37.8) circle (3);
\draw (47.3,-37.8) node {$3$};
\draw [black] (34.3,-37.8) circle (3);
\draw (34.3,-37.8) node {$4$};
\draw [black] (29.5,-26) circle (3);
\draw (29.5,-26) node {$5$};
\draw [black] (42.6,-19.3) -- (49.1,-24.2);
\draw [black] (50.49,-28.83) -- (48.31,-34.97);
\draw [black] (44.3,-37.8) -- (37.3,-37.8);
\draw [black] (33.17,-35.02) -- (30.63,-28.78);
\draw [black] (31.85,-24.13) -- (37.85,-19.37);
\end{tikzpicture}
    \caption{The cycle graph $C_5$.}
    \label{fig:five-cycle}
\end{minipage}%
\begin{minipage}{.55\textwidth}
     \begin{center}
\begin{tikzpicture}[scale=0.128]
\tikzstyle{every node}+=[inner sep=0pt]
\draw [black] (45.5,-19.4) circle (3);
\draw (45.5,-19.4) node {$2$};
\draw [black] (29.6,-19.4) circle (3);
\draw (29.6,-19.4) node {$1$};
\draw [black] (29.6,-33.5) circle (3);
\draw (29.6,-33.5) node {$3$};
\draw [black] (45.5,-33.5) circle (3);
\draw (45.5,-33.5) node {$4$};
\draw [black] (42.5,-19.4) -- (32.6,-19.4);
\draw [black] (43.26,-21.39) -- (31.84,-31.51);
\draw [black] (45.5,-22.4) -- (45.5,-30.5);
\end{tikzpicture}
\end{center}
    \caption{A 112-representable graph.}
    \label{fig:112-rep}
\end{minipage}
\end{figure}

\begin{example}
The graph shown in Figure \ref{fig:112-rep} is 112-representable with $w = 121334$ serving as a 112-representant.
\end{example}

\begin{remark}
The word-representable graphs are precisely the 11-representable graphs. This can be seen by noting that ``alternating" is equivalent to avoiding the word 11. Additionally, a word $1^k$-represents a graph if and only if it $a^k$-represents it. 
\end{remark}

\begin{remark}
In \cite{K4}, Kitaev defines a similar notion of graph representability associated with a pattern $u$ on two letters, which we will call $u$\textit{-Kitaev-representability}. Kitaev's notion aligns with Definition \ref{def:t-rep}, except that his notion depends on an ordering of the vertices which must match the ordering of letters comprising $u$. For instance, $112$-Kitaev-representants are distinct from $221$-Kitaev-representants.  In particular, Kitaev's version is defined for labeled graphs, while Definition \ref{def:t-rep} works for unlabeled graphs.
\end{remark}

\begin{example}
Let $w=2123$. Then $w$ 121-represents the graph with vertex set $\{1, 2, 3\}$ with edges between $1$ and $3$ and between $2$ and $3$. However, $w$ 121-Kitaev-represents the complete graph on the vertices $\{1, 2,3\}$.
\end{example}

We now prove several results on the representability of graphs for various $t$. In \cite{K4}, Kitaev shows that for all $u$ of length at least three, every graph is $u$-Kitaev-representable. Kitaev's approach is to first note that the complete graph $K_n$ is $u$-Kitaev-representable, and then to show that if $w$ $u$-Kitaev-represents any graph $G$, one can construct a word $w'$ from $w$ that $u$-Kitaev-represents $G'$, where $G'$ is the graph $G$ after deleting some edge. We use this same approach to show that all graphs are $t$-representable for certain $t$.

Note that if $t$ is of the form $1^k$ for some $k>1$, then $t$-representants are the same as $u$-Kitaev-representants. Thus we only consider the cases in which $t$ has two distinct letters. 

\begin{theorem}
\label{thm:t-rep-1}
If $t$ is of the form $a^kb^la$ for positive integers $k$ and $l$, then every graph is $t$-representable.
\end{theorem}

\begin{proof}
Note that $K_n$ is $t$-representable by the word $123\cdots n$. Let $G$ be any graph, and suppose that $w$ $t$-represents $G$. Let $ij$ be an edge in $G$ and define $G'$ as the graph obtained from $G$ by deleting the edge $ij$.

Let $t[i, j]$ be the word $t$ after the substitution $a\rightarrow i, b\rightarrow j$ has been made.  Let $r(\sigma(w))_k$ denote the $k^\text{th}$ letter in the reverse of the final permutation of $w$ and let $\sigma(w)_k^l$ denote $l$ consecutive instances of the letter $\sigma(w)_k$.   Take
$$w'=wr(\sigma(w))_1^{l+1}r(\sigma(w))_2^{l+1}\cdots r(\sigma(w))_n^{l+1}t[i,j].$$ 
We claim that $w'$ $t$-represents $G'$. 

Indeed, because $w'$ contains $w$, only the edges in $G$ will be included, and because $w'$ includes $t[i, j]$, there will be no edge between $i$ and $j$. Thus it suffices to show that no other edges are removed.

Assume there exists an instance of the pattern $c^kd^lc$ in $w'$ that is not in $w$, where $\{c, d\}\neq \{i, j\}$. Then consider where the $d^l$ can be placed. It cannot be a part of in the section $r(\sigma(w))_1^{l+1}r(\sigma(w))_2^{l+1}\cdots r(\sigma(w))_n^{l+1}$, or else there would be $l+1$ consecutive appearances of $d$. Additionally, it cannot come in $t[i, j]$, or else $\{c, d\} = \{i, j\}$. Therefore, it must be completely contained in $w$. It follows that there can be no instance of $c$ in $w$ that comes after those $l$ $d$'s, or else the pattern would have already existed in $w$. Therefore, $c$ must appear before $d$ in $w'$ after the prefix $w$. But by our construction, the opposite happens. Thus there is no such instance of $c^kd^lc\cdots$ in $w'$, so no new edges other than $ij$ are removed.  Thus $w'$ $t$-represents $G$, as desired.  

In this way, we can construct a word $t$-representing $G$ for any graph $G$ by removing edges one at a time from $K_n$ to reach $G$.
\end{proof}

\begin{theorem}
\label{thm:t-rep-2}
If $t$ is of the form $a^kb^l$ where $k, l\ge 2$, then every graph is $t$-representable.
\end{theorem}
\begin{proof}
Note again that $K_n$ is $t$-representable by the word $123\cdots n$. Now let $G$ be any graph, and suppose $w$ $t$-represents $G$. Let $ij$ be an edge in $G$ and let $G'$ be the graph obtained from $G$ by deleting the edge $ij$. Take $w'=w\sigma(w)t[i,j]$. It suffices to show that $w'$ $t$-represents $G'$.  

Indeed, because $w'$ contains $w$, only the edges in $G$ will be included, and because $w'$ includes $t[i, j]$, there will be no edge between $i$ and $j$. Thus it suffices to show that no other edges are removed.

Assume there exists an instance of the pattern $c^kd^l$ in $w'$ that is not in $w$, where $\{c, d\}\neq \{i, j\}$. If $c^kd^l$ begins in the $w$ section, then the entirety of $c^k$ must be contained in $w$, or else a $d$ will appear before all $k$ $c$'s appear in a row. Now if the $d$'s after $c^k$ begins at the end of $w$, then $c$ comes before $d$ in $\sigma (w)$, which does not work either. Thus $c^kd^l$ has to appear after $w$ in $w'$. Then it has to begin in $\sigma(w)$, but there is only one appearance of $c$ in $\sigma(w)$. Thus both $c$ and $d$ appear in $t[i, j]$, and thus $\{c, d\}=\{i, j\}$, contradiction. Thus there is no such instance of $c^kd^l$ in $w'$, so no new edges other than $ij$ are removed. It follows that $w'$ $t$-represents $G$, as desired.  
\end{proof}

The remaining cases are of the form $t=a^kb$ and $t=ab^k$, where $k\ge 1$. First, note that the $t$-representable graphs in these cases are the same. Indeed, if $w$ $a^kb$-represents some graph $G$, then the reverse of $w$ $ab^k$-represents $G$ and vice versa.  

For the case of $k=1$, because every vertex must appear at least once in a $t$-representant, there will always be an appearance of the pattern $ab$ for any two vertices. Thus, only empty graphs are $ab$-representable. We now address the case $k\ge 3$.  

\begin{theorem}
\label{thm:t-rep-3}
Let $k\ge 3$ and set $t=a^kb$. Then every graph is $t$-representable.
\end{theorem}
\begin{proof}
Note again that $K_n$ is $t$-representable by the word $123\cdots n$. Now let $G$ be any graph, and suppose $w$ $t$-represents $G$. Let $ij$ be an edge in $G$ and let $G'$ be the graph obtained from $G$ by deleting the edge $ij$.

Let $v$ be a permutation of the set $V(G)\backslash \{i, j\}$. Take $w'=i^{k-1}vij\pi(w)w$. It suffices to show that $w'$ $t$-represents $G'$.  

Indeed, because $w'$ contains $w$, only the edges in $G$ will be included, and because $w'$ includes an instance of $i^kj$ at the front, there will be no edge between $i$ and $j$. Thus it suffices to show that no other edges are removed.

Assume there exists an instance of the pattern $c^kd^l$ in $w'$ that is not in $w$, where $\{c, d\}\neq \{i, j\}$. Clearly the $c^k$ section cannot intersect with the $i^{k-1}$ section, so it must begin in the section $vij\pi(w)$. Note that this is a concatenation of two permutations of the set $V(G)$. Thus there can only be two occurrences of $c$ in that section. But because $\pi(w)$ is the initial permutation of $w$, every other element will appear between the second occurrence of $c$ and its first occurrence in $w$, so we cannot have an instance of $c^kd$. Thus there is no such instance of $c^kd^l$ in $w'$, so no new edges other than $ij$ are removed. It follows that $w'$ $t$-represents $G$, as desired. 
\end{proof}

For $k=2$, we have not yet fully determined the $a^kb$-representable graphs. However, we have the following theorem.

\begin{theorem}
\label{thm:t-rep-4}
All $aa$-representable graphs (i.e. word-representable graphs) are also $aab$-representable.
\end{theorem}
\begin{proof}
Let $w$ $aa$-represent a graph $G$. We claim that $w \sigma(w)$ $aab$-represents $G$, where we recall that $\sigma(w)$ denotes the final permutation of $w$. Let $x$ and $y$ be any two distinct vertices of $G$. 

Suppose first that $(x,y) \not \in E(G)$. Then the subword $w\vert_{xy}$ of $w$ induced by $x$ and $y$ contains at least one occurrence of the pattern $aa$. If $w \vert_{xy}$ also contains at least one occurrence of $aab$, then so does $w \sigma(w)$, and this direction of the proof is complete. Therefore, suppose that $w \vert_{xy}$ avoids $aab$. Then $w \vert_{xy}$ must end in the pattern $aa$. If $w \vert_{xy}$ ends in $xx$, then $\sigma(w) \vert_{xy} = yx$, meaning $(w \sigma(w) )\vert_{xy}$ ends in $xxyx$, which contains the pattern $aab$. The case in which $w \vert_{xy}$ ends in $yy$ follows similarly.  

Next, suppose that $(x,y) \in E(G)$. Then $w \vert_{xy}$ avoids $aa$ (and therefore avoids $aab$). Appending $\sigma(w)$ to $w$ does not introduce any occurrences of the pattern $aa$. Therefore, $(w \sigma(w)) \vert_{xy}$ avoids $aab$. 
\end{proof}

\section{Future work} \label{sec:future work}
In Section \ref{sec:minimal length}, we proved the following upper bound for $\ell(G)$ in terms of its connected components:
\begin{equation}
\label{ell-bound}
\ell(G) \le \sum_{i=1}^k (\ell(G_i) + |V_i|) - \max_{1\le j\le k}{|V_j|}.
\end{equation}

It would be interesting to strengthen this bound. This leads to the following problem.

\begin{problem}
Find stricter bounds for $\ell(G)$ in terms of its connected components and classify all graphs for which equality in Equation \ref{ell-bound} holds.
\end{problem}

We also proved that for both trees and cycles, $\ell(G)=2|G|-2$.  We also showed that for triangle-free graphs $G$, $\ell(G)\ge 2|G|-2$. We ask if these observations can be extended to classify graphs with representation number 2 for which $\ell(G)$ is close to $2|G|$.

\begin{problem}
Classify all graphs $G$ with representation number 2 for which 
\begin{itemize}
    \item [(a)] $\ell(G) = 2|G|-2$, 
    \item [(b)] $\ell(G) = 2|G|-1$, 
    \item [(c)] $\ell(G) = 2|G|$.
\end{itemize}
\end{problem}

\noindent In \cite{K1}, Kitaev classifies all graphs with representation number 2. This gives a nice starting point from which to attack the problem above.

Also in Section \ref{sec:minimal length}, we showed that $n(P_k)  = (k+1)\cdot 2^{k-3}$ for $k \geq 3$. It is worth noting that $(k+1) \cdot 2^{k-3}$ is also a formula for sequence A079859 in the On-line Encyclopedia of Integer Sequences (OEIS) \cite{oeis}. It could be interesting to investigate whether the minimal length word-representants of $P_k$ relate in any natural way to the other objects enumerated by this same formula.
\begin{problem}
Find a natural bijection between the minimal length word-representants of $P_k$ ($k \geq 3$) and the other objects enumerated by OEIS sequence A079859.
\end{problem}

In Section \ref{sec-patterns}, we proved that all graphs are $t$-representable for several choices of words $t$.  Finding which graphs are representable for a general word may be quite complicated and difficult.  However, finding the answer for the case of $t=ab^2$ or $a^2b$, which are equivalent by reflection, is still open and may be tractable.

\begin{problem}
Characterize the graphs which are $a^2b$-representable.
\end{problem}

\section*{Acknowledgments}
This research was conducted at the University of Minnesota Duluth REU, funded by NSF Grant
1659047 and NSA Grant H98230-18-1-0010. The authors would like to thank Joe Gallian for running the REU and suggesting the topic.  The authors also thank Amanda Burcroff and anonymous reviewers for many helpful comments.

\end{document}